\documentclass[12pt]{amsart} 

\usepackage{amsmath, amssymb, amsfonts, amsbsy, amsthm, latexsym, stmaryrd, mathtools, fullpage, enumerate, ytableau, shuffle, multirow, mathdots, enumerate, tikz}

\usepackage{amsthm}
\usepackage{comment}

\newtheorem{thm}{Theorem}[section]
\newtheorem{lem}[thm]{Lemma}
\newtheorem{prop}[thm]{Proposition}
\newtheorem{cor}[thm]{Corollary}
\theoremstyle{definition}
\newtheorem{ex}[thm]{Example}
\newtheorem{remark}[thm]{Remark}
\newtheorem{defn}[thm]{Definition}
\newtheorem*{defn*}{Definition}
\newtheorem*{thm*}{Theorem}
\newtheorem*{prop*}{Proposition}
\newtheorem*{cor*}{Corollary}

\def\Tab{P_{SK}}  
\def\rTab{Q_{SK}} 
\def\ksw{\hat{\equiv}}  
\def\pr{{^{\prime}}}
\def\row{\mathfrak{row}}
\def\insert{\xleftarrow{SK}}
\def\D{\mathcal{D}}

\def\std{\operatorname{std}}
\def\st{\mathfrak{st}}
\def\symfcn{K_{\lambda}}
\newcommand{\sss}{\scriptstyle}
\newcommand{\mw}{\mathcal{W}}

\newcommand{\mo}{\mathcal{O}}

\begin{document}
\title{Shifted Hecke insertion and the $K$-theory of $OG(n,2n+1)$}
\author{Zachary Hamaker, Adam Keilthy, Rebecca Patrias, Lillian Webster, Yinuo Zhang, Shuqi Zhou}
\date{}
\begin{abstract}
Patrias and Pylyavskyy introduced shifted Hecke insertion as an application of their theory of dual filtered graphs.
We use shifted Hecke insertion to construct symmetric function representatives for the $K$-theory of the orthogonal Grassmannian.
These representatives are closely related to the shifted Grothendieck polynomials of Ikeda and Naruse.
We then recover the $K$-theory structure coefficients of Clifford-Thomas-Yong/Buch-Samuel by introducing a shifted $K$-theoretic Poirier-Reutenauer algebra.
Our proofs depend on the theory of shifted $K$-theoretic jeu de taquin and the weak $K$-Knuth relations.

\end{abstract}
\maketitle

\section{Introduction}

In~\cite{PPshifted}, Patrias and Pylyavskyy introduce shifted Hecke insertion as an application of their theory of dual filtered graphs.
It is a bijection between finite words in the positive integers and pairs $(\Tab,\rTab)$ of shifted tableaux of the same shape, where $\Tab$ is increasing and $\rTab$ is set-valued.
Shifted Hecke insertion is the shifted analogue of Hecke insertion, a bijection introduced by Buch, Kresch, Shimozono, Tamvakis, and Yong to derive a Littlewood-Richardson rule for the $K$-theory of Grassmannians~\cite{HeckeInsertion}.
The primary task of this paper is to relate shifted Hecke insertion to the $K$-theory of the orthogonal Grassmannian $OG(n,2n+1)$.

The $K$-theory of the orthogonal Grassmannian is well understood.
It has as a basis $\mo_\lambda$ of Schubert structure sheaves indexed by shifted shapes.
The product structure is determined by a combinatorial Littlewood-Richardson rule
\[
\mo_\lambda \cdot \mo_\mu = \sum_{\nu} (-1)^{|\nu| - |\lambda| - |\mu|} c^\nu_{\lambda, \mu} \mo_\nu
\]
first proven by Clifford, Thomas, and Yong~\cite{clifford2014k} using shifted $K$-theoretic jeu de taquin and a Pieri rule of Buch and Ravikumar~\cite{buch2012pieri}. It was later proven in more generality by Buch and Samuel~\cite{BuchandSamuel2013}.
Moreover, the factorial Schur $P$-functions of Ikeda and Naruse are equivariant $K$-theory representatives~\cite{ikeda2013k}.
These can be specialized to the \emph{shifted stable Grothendieck polynomials} $GP_\lambda$, which form representatives for each $\mo_\lambda$.
We use shifted Hecke insertion to construct an alternative combinatorial approach to these results.

For each shifted shape $\lambda$, we define the \emph{weak shifted stable Grothendieck polynomial} $K_\lambda$ as a generating function over multiset-valued shifted tableaux.
We demonstrate an equivalent definition of $K_\lambda$ using shifted Hecke insertion.
\begin{thm}
\label{t:wssG}
The weak shifted stable Grothendieck polynomial $K_\lambda$ is symmetric.

\end{thm}

The proof of Theorem~\ref{t:wssG} proceeds by relating shifted Hecke insertion to shifted $K$-jeu de taquin.
In doing so, we show that shifted Hecke insertion respects the \emph{weak $K$-Knuth equivalence} of Buch and Samuel~\cite{BuchandSamuel2013}.
This allows us to expand $K_\lambda$ in terms of the (signless) \emph{stable Grothendieck polynomials} $G_\mu$, which are symmetric functions that form $K$-theory representatives for the Grassmannian.
We will see in Proposition~\ref{prop:Ikeda} that the $K_\lambda$ and the $GP_\lambda$ of Ikeda and Naruse are closely related.
As a corollary, we derive a new proof of the symmetry of $GP_\lambda$.
Surprisingly, while $GP_\lambda$ satisfies the $Q$-cancellation property, we show that $K_\lambda$ does not, hence cannot be expressed as a sum of Schur $P$-functions.

Using weak $K$-Knuth equivalence, we define a shifted $K$-theoretic Poirier-Reutenauer algebra.
This generalizes the shifted Poirier-Reutenauer algebra of Jang and Li~\cite{JLshifted}, a shifted analogue of the Poirier-Reutenauer Hopf algebra~\cite{poirier1995algebres}.
Our approach follows work of Patrias and Pylyavskyy on a $K$-theoretic Poirier-Reutenauer bialgebra~\cite{PPK}.
The shifted $K$-theoretic Poirier-Reutenauer algebra does not have the coalgebra structure analogous to that of shifted Poirier-Reutenauer and $K$-theoretic Poirier Reutenauer. 
Using the shifted $K$-theoretic Poirier-Reutenauer algebra, we define a Littlewood-Richardson rule for the product $K_\lambda \cdot K_\mu$.
This rule coincides up to sign with the rule of Clifford, Thomas, and Yong~\cite{clifford2014k} and Buch and Samuel~\cite{BuchandSamuel2013}.
As a consequence, we have our main result.
\begin{thm}
\label{t:ktheory}
Each $K_\lambda$ is a representative for $\mo_\lambda$.

\end{thm}

The remainder of the paper is structured as follows.
In the next section, we define shifted Hecke insertion and show how it relates to shifted $K$-theoretic jeu de taquin.
This allows us to demonstrate the relationship between shifted Hecke insertion and the weak $K$-Knuth equivalence relations.
The third section is devoted to defining the $K_\lambda$, expressing them in terms of shifted Hecke insertion and proving Theorem~\ref{t:wssG}.
In the fourth and final section, we develop the shifted $K$-theoretic Poirier-Reutenauer algebra and prove Theorem~\ref{t:ktheory}.

\section{Shifted Hecke Insertion and Weak $K$-Knuth Equivalence}
 
We show that the shifted Hecke insertion given in \cite{PPshifted} respects the weak $K$-Knuth equivalence given in \cite{BuchandSamuel2013}. 
Before presenting our argument, we review previous work on increasing shifted tableaux, shifted Hecke insertion, and shifted $K$-jeu de taquin.

\subsection{Increasing shifted tableaux}
To each strict partition $\lambda = (\lambda_1 > \lambda_2 > \ldots > \lambda_k)$ we associate the \emph{shifted shape}, which is an array of boxes where the $i$th row has $\lambda_i$ boxes and is indented $i-1$ units.
A \emph{shifted tableau} is a filling of the shifted shape with positive integers.
A shifted tableau is \emph{increasing} if the labels are strictly increasing from left to right along rows and top to bottom down columns. 
The \textit{reading word} of an increasing shifted tableau $T$, denoted $\row(T)$, is the word obtained by reading the entries left to right from the bottom row to the top row.
For example, the first two shifted tableaux below are increasing, while the third one is not.
The first two shifted tableaux have reading words $635124$ and $8471367$, respectively.

\begin{center}
\begin{ytableau}
1 & 2 & 4 \\
\none & 3 & 5 \\
\none & \none& 6 
\end{ytableau} \hspace{.5in}
\begin{ytableau}
1 & 3 & 6 & 7 \\
\none & 4 & 7 \\
\none & \none & 8 
\end{ytableau} \hspace{.5in}
\begin{ytableau}
1 & 2 & 4 & 6\\
\none & 3 & 4 \\
\none & \none & 5 
\end{ytableau} 
\end{center}

\begin{lem}\label{finitetableau}
There are only finitely many increasing shifted tableaux filled with a given finite alphabet.

\end{lem}
\begin{proof}
If the alphabet has $n$ letters, each row and column of the tableau can be no longer than $n$.  
\end{proof}

\subsection{Shifted Hecke Insertion}
We now restate the rules for shifted Hecke insertion, introduced by Patrias and Pylyavskyy in~\cite{PPshifted}.
It is simultaneously a shifted analogue of Buch, Kresch, Shimozono, Tamvakis and Yong's Hecke insertion~\cite{HeckeInsertion} and a $K$-theoretic analogue of Sagan-Worley insertion, due independently to Sagan and Worley~\cite{journals/jct/Sagan87, worley1984theory}. 
From this point on, ``insertion" will always refer to shifted Hecke insertion unless stated otherwise. 

First, we describe how to insert a positive integer $x$ into a given shifted increasing tableau $T$, with examples interspersed for clarity. 
We start by inserting $x$ into the first row of $T$. 
For each insertion, we assign a box to record where the insertion terminates. 
This data will be used when we define the recording tableau in Section~\ref{svst}.  

The rules for inserting $x$ into a row or column of $T$ are as follows:
\begin{enumerate}
\item[(1)] If $x$ is weakly larger than all integers in the row (resp. column) and adjoining $x$ to the end of the row (resp. column) results in an increasing tableau $T'$, then $T'$ is the resulting tableau. 
We say the insertion terminates at the new box. 
\end{enumerate}
\noindent Inserting 5 into the first row of the left tableau gives us the right tableau below. The insertion terminates at box $(1,4)$.
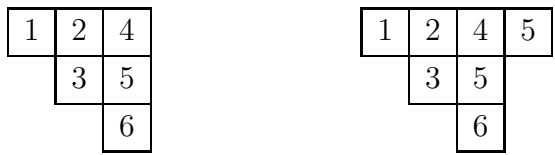
\begin{center}
\begin{ytableau}
1 & 2 & 4 \\
\none & 3 & 5 \\
\none & \none & 6 
\end{ytableau} \hspace{1in}
\begin{ytableau}
1 & 2 & 4 & 5 \\
\none & 3 & 5 \\
\none & \none & 6 
\end{ytableau}
\end{center}

\begin{enumerate}
\item[(2)] If $x$ is weakly larger than all integers in the row (resp. column) and adjoining $x$ to the end of the row (resp. column) does not result in an increasing tableau, then $T' = T$. 
If $x$ is row inserted into a nonempty row, we say the insertion terminated at the box at the bottom of the column containing the rightmost box of this row. 
If $x$ is row inserted into an empty row, we say that the insertion terminated at the rightmost box of the previous row. 
If $x$ is column inserted, we say the insertion terminated at the rightmost box of the row containing the bottom box of the column $x$ could not be added to. 
\end{enumerate}

\noindent Adjoining 5 to the second row of the tableau on the left does not result in an increasing tableau. Thus the insertion of 5 into the second row of the tableau on the left terminates at (2,3) and gives us the tableau on the right.
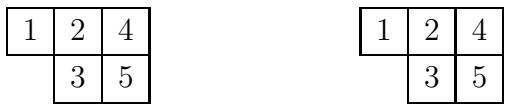
\begin{center}
\begin{ytableau}
1 & 2 & 4\\
\none & 3 & 5
\end{ytableau} \hspace{1in}
\begin{ytableau}
1 & 2 & 4\\
\none &3 & 5
\end{ytableau}
\end{center}
Adjoining 2 to the (empty) second row of the tableau below does not result in an increasing tableau. The insertion ending in this failed row insertion terminates at $(1,3)$.
\begin{center}
\begin{ytableau}
1 & 2 & 3
\end{ytableau}
\end{center}
Adjoining 3 to the end of the third column of the left tableau does not result in an increasing tableau. This insertion terminates at $(1,3)$.
\begin{center}
\begin{ytableau}
1 & 2 & 3 \\
\none & 3 
\end{ytableau} 
\end{center}

For the next two rules, suppose the row (resp. column) contains a box with label strictly larger than $x$, and let $y$ be the smallest such box.
\begin{enumerate}
\item[(3)]If replacing $y$ with $x$ results in an increasing tableau, then replace $y$ with $x$. 
In this case, $y$ is the output integer. 
If $x$ was inserted into a column or if $y$ was on the main diagonal, proceed to insert all future output integers into the next column to the right. 
If $x$ was inserted into a row and $y$ was not on the main diagonal, then insert $y$ into the row below.
\end{enumerate}
\noindent Row inserting 3 into the first row of the tableau on the left results in the tableau below on the right. 
This insertion terminates at $(2,3)$.
\begin{center}
\begin{ytableau}
1 & 2 & 4 \\
\none & 3 
\end{ytableau}\hspace{1in}
\begin{ytableau}
1 & 2 & 3 \\
\none & 3 & 4
\end{ytableau} 
\end{center}
To insert 3 into the second row of the tableau below on the left, replace 4 with 3, and column insert 4 into the third column. 
The resulting tableau is on the right.
\begin{center}
\begin{ytableau}
1 & 2 & 3 & 5 \\
\none & 4 & 6 \\
\none & \none & 8 
\end{ytableau} \hspace{1in}
\begin{ytableau}
1 & 2 & 3 & 5 \\
\none & 3 & 4 & 6 \\
\none & \none & 8 
\end{ytableau}
\end{center}

\begin{enumerate}
\item[(4)]If replacing $y$ with $x$ does not result in an increasing tableau, then do not change the row (resp. column). 
In this case, $y$ is the output integer. 
If $x$ was inserted into a column or if $y$ was on the main diagonal, proceed to insert all future output integers into the next column to the right. 
If $x$ was inserted into a row, then insert $y$ into the row below.
\end{enumerate}

\noindent If we insert 3 into the first row of the tableau below, notice that replacing 5 with 3 does not create an increasing tableau. 
Hence row insertion of 3 into the first row produces output integer 5, which is inserted into the second row. 
Replacing 6 with 5 in the second row does not create an increasing tableau. This produces output integer 6. 
Adjoining 6 to the third row does not result in an increasing tableau. 
Thus inserting 3 into the tableau below does not change the tableau. 
This insertion terminates at $(2,3)$. 
\begin{center}
\begin{ytableau}
1 & 3 & 5 \\
\none & 4 & 6 \\
\end{ytableau} 
\end{center}

For any given word $w = w_1w_2\cdots w_n$, we define the \textit{insertion tableau} $\Tab(w)$ of $w$ to be $(\cdots((\emptyset\insert w_1)\insert w_2)\cdots\insert w_n)$, where $\emptyset$ denotes the empty shape and $\insert$ denotes the insertion of a single letter. 
For example, the sequence of tableaux obtained  when inserting $2115432$ is
\begin{center}
\begin{ytableau}
2 \end{ytableau}\hspace{.1in}
\begin{ytableau}
1 & 2 \end{ytableau}\hspace{.1in}
\begin{ytableau}
1 & 2 \end{ytableau}\hspace{.1in}
\begin{ytableau}
1 & 2  &5 \end{ytableau}\hspace{.1in}
\begin{ytableau}
1 & 2 & 4 \\ \none &5\end{ytableau}\hspace{.1in}
\begin{ytableau}
1 & 2 & 3 \\\none & 4 & 5 \end{ytableau}\hspace{.1in}
\begin{ytableau}
1 & 2 & 3 & 5\\ \none & 3 & 4 \end{ytableau}
\end{center}
 The tableau on the right is $\Tab(2115432)$.

For any interval $I$, we define $T|_{I}$ to be the tableau obtained from $T$ by deleting all boxes with labels not in $I$ and $w|_{I}$ to be the word obtained from $w$ by deleting all letters not in $I$. 
We use $[k]$ to denote the interval $\left\{1,2,...,k\right\}$. 
The following simple lemma will be useful. 

\begin{lem}\label{intervaltableau}
If $\Tab(w) = T$, then $\Tab(w)|_{[k]} = \Tab(w|_{[k]}) = T|_{[k]}$.
\end{lem}
\begin{proof}
Observe from from the insertion rules that letters greater than $k$ never affect the placement or number of letters in $\{1,2,\ldots,k\}$.  
\end{proof}

\subsection{Recording tableaux} \label{svst}

In order to describe the recording tableau for shifted Hecke insertion of a word $w$, we need the following definition. 

\begin{defn}\cite{ikeda2013k, PPshifted}\label{def:setvaluedshifted} A \textit{set-valued shifted tableau} is defined to be a filling of the boxes of a shifted shape with finite, nonempty subsets of primed and unprimed positive integers with ordering $1'<1<2'<2<\ldots$ such that:

\begin{enumerate}
\item[1)] The smallest number in each box is greater than or equal to the largest number in the box directly to the left of it, if that box exists.
\item[2)] The smallest number in each box is greater than or equal to the largest number in the box directly above it, if that box exists.
\item[3)] There are no primed entries on the main diagonal.
\item[4)] Each unprimed integer appears in at most one box in each column.
\item[5)] Each primed integer appears in at most one box in each row. 
\end{enumerate}
\end{defn}

A set-valued shifted tableau is called \textit{standard} if the set of labels is exactly $[n]$ for some $n$, each appearing either primed or unprimed exactly once. 
For example, the tableaux below are set-valued shifted tableaux. The tableau on the right is standard.

\begin{center}
\ytableausetup{boxsize=.7cm}
\begin{ytableau}
1 & 2 3\pr & 6\pr 6 \\
\none & 4 & 8\pr9
\end{ytableau}\hspace{1in}
\begin{ytableau}
1 & 2 & 3\pr 4\pr & 6\pr \\
\none & 5\\
\end{ytableau} 
\end{center}

The \textit{recording tableau} of a word $w=w_1w_2\ldots w_n$, denoted  $\rTab(w)$, is a standard set-valued shifted tableau that records where the insertion of each letter of $w$ terminates. 
We define it inductively. 

Start with $\rTab(\emptyset) = \emptyset$. 
If the insertion of $w_k$ added a new box to $\Tab(w_1w_2\ldots w_{k-1})$, then add the same box  with label $k$ ($k\pr$ if this box was added by column insertion) to $\rTab(w_1w_2\ldots w_{k-1})$. 
If $w_k$ did not change the shape of $\Tab(w_1w_2\ldots w_{k-1})$, we obtain $\rTab(w_1w_2\ldots w_k)$ from $\rTab(w_1\ldots w_{k-1})$ by adding the label $k$ ($k\pr$ if it ended with column insertion) to the box where the insertion terminated. 
If insertion terminated when a letter failed to insert into an empty row, label the box where the insertion terminated $k\pr$.

\begin{ex}\label{ex:word}
Let $w = 451132$.  
We insert $w$ letter by letter, writing the insertion tableau at each step in the top row and the recording tableau at each step in the bottom row.  \\

\begin{center}
\ytableausetup{boxsize=.6cm}
\begin{ytableau}
4
\end{ytableau} \hspace{.15in}
\begin{ytableau}
4 & 5 
\end{ytableau} \hspace{.15in}
\begin{ytableau}
1 & 4 & 5 
\end{ytableau} \hspace{.15in}
\begin{ytableau}
1 & 4 & 5 
\end{ytableau} \hspace{.15in}
\begin{ytableau}
1 & 3 & 5 \\
\none & 4
\end{ytableau} \hspace{.15in}
\begin{ytableau}
1 & 2 & 4 & 5 \\
\none & 3
\end{ytableau}
$=\Tab(w)$ 
\end{center}

 \begin{center}
\begin{ytableau}
1
\end{ytableau} \hspace{.15in}
\begin{ytableau}
1 & 2 
\end{ytableau} \hspace{.15in}
\begin{ytableau}
1 & 2 & 3\pr 
\end{ytableau} \hspace{.15in}
\begin{ytableau}
1 & 2 & 3\pr 4\pr 
\end{ytableau} \hspace{.15in}
\begin{ytableau}
1 & 2 & 3\pr 4\pr \\
\none & 5 
\end{ytableau} \hspace{.15in}
\begin{ytableau}
1 & 2 & 3\pr 4\pr & 6\pr \\
\none & 5
\end{ytableau}
$=\rTab(w)$ 
\end{center}
\end{ex}

In \cite{PPshifted}, Patrias and Pylyavskyy define a reverse insertion procedure so that one can recover a word $w$ for each pair $(P,Q)$ consisting of an increasing shifted tableaux and a standard set-valued shifted tableau of the same shape.
See \cite{PPshifted} for details on reverse shifted Hecke insertion.  
This procedure gives the following result:  

\begin{thm} \emph{\cite[Theorem~5.19]{PPshifted}}
\label{wordbijection}
The map $w \mapsto (\Tab(w),\rTab(w))$ is a bijection between words of positive integers and pairs of shifted tableaux $(P,Q)$ of the same shape where $P$ is an increasing shifted tableau and $Q$ is a standard set-valued shifted tableau.
\end{thm}

\subsection{Weak $K$-Knuth equivalence}
The Knuth relations determine which words have the same Robinson-Schensted-Knuth insertion tableau \cite[Theorem A1.1.4]{opac-b1096212}. 
We present Buch and Samuel's shifted $K$-theoretic analogue, called weak $K$-Knuth equivalence~\cite{BuchandSamuel2013}. 
As we will see in Theorem~\ref{insertionimplieskknuth} and Example~\ref{notURT}, weak $K$-Knuth equivalence is a necessary but not sufficient condition for two words to have the same shifted Hecke insertion tableau. 

\begin{defn}\cite[Definition~7.6]{BuchandSamuel2013}
\label{def:wkknuth} Define the \textit{weak $K$-Knuth equivalence relation} on the alphabet \{1,2,3,$\cdots$\}, denoted by $\ksw$, as the symmetric transitive closure of the following relations, where \textbf{u} and \textbf{v} are (possibly empty) words of positive integers, and  $a < b < c$ are distinct positive integers:  
\begin{enumerate}
\item (\textbf{u},$a,a$,\textbf{v}) $\ksw$ (\textbf{u},$a$,\textbf{v}),
\item (\textbf{u},$a,b,a$,\textbf{v}) $\ksw$ (\textbf{u},$b,a,b$,\textbf{v}),
\item (\textbf{u},$b,a,c$,\textbf{v}) $\ksw$ (\textbf{u},$b,c,a$,\textbf{v}),
\item (\textbf{u},$a,c,b$,\textbf{v}) $\ksw$ (\textbf{u},$c,a,b$,\textbf{v}),
\item ($a,b$,\textbf{u}) $\ksw$ ($b,a$,\textbf{u}).
\end{enumerate}
\end{defn}

Two words $w$ and $w'$ are \emph{weak $K$-Knuth equivalent}, denoted $w\ \ksw \ w'$, if $w'$ can be obtained from $w$ by a finite sequence of weak $K$-Knuth equivalence relations.
Two shifted increasing tableaux $T$ and $T'$ are weak $K$-Knuth equivalent if $\row(T)$ $\ksw$ $\row(T')$.
For example, $1243$ $\ksw$ $442143$ since
\[
1243\text{ }\ksw\text{ } 2143\text{ } \ksw \text{ }21143\text{ } \ksw \text{ }21413 \text{ }\ksw\text{ } 24143 \text{ }\ksw\text{ } 42143\text{ } \ksw \text{ }442143.
\]

By removing relation (5), we obtain the $K$-Knuth relations of Buch and Samuel~\cite{BuchandSamuel2013}.
As in the $K$-Knuth equivalence but in contrast to Knuth equivalence, each weak $K$-Knuth equivalence class has infinitely many elements and contains words of arbitrary length. 

We will need the following lemma, which follows easily from the weak $K$-Knuth relations. 

\begin{lem}\label{interval}
If $w$ $\ksw$ $w'$, then $w\arrowvert_{I}$ $\ksw$ $w'\arrowvert_{I}$ for any interval $I$.
\end{lem}

\subsection{Shifted $K$-theoretic jeu de taquin} 

We next describe the shifted $K$-theoretic jeu de taquin algorithm or \textit{shifted $K$-jdt} introduced by Clifford, Thomas, and Yong in~\cite{clifford2014k}. From this point on, ``jeu de taquin'' will refer to shifted $K$-jdt.

\begin{defn} Given a shifted skew shape $\lambda/\mu$ with $\mu\subset\lambda$ strict partitions, we say that two boxes in $\lambda/\mu$ are \textit{adjacent} if they share a common edge. 
We call a box in $\lambda/\mu$ \textit{maximal} if it has no adjacent boxes to the east or south and \textit{minimal} if it has no adjacent boxes to the west or north.
\end{defn}

For the following definition, let $T(\alpha)$ denote the label of box $\alpha$ in tableau $T$.

\begin{defn}\cite[Section 4]{ThomasYong} We define the action of the $K$-theoretic switch operator $(i, j)$ on a tableau $T$ as follows:
\[
((i,j)T)(\alpha) = 
\begin{cases}
 j &\text{if } T(\alpha)=i \text{ and } T(\beta)=j 
 \text{ for some box }\beta \text{ adjacent to } \alpha; \\ 
 i &\text{if } T(\alpha)=j \text{ and } T(\beta)=i \text{ for some box }\beta\text{ adjacent to } \alpha; \\ 
 T(\alpha) &\text{otherwise}.
\end{cases}
\]
In other words, we swap the labels of boxes labeled by $i$ and the adjacent box (or boxes) labeled with entry $j$.  If no box labeled $j$ is adjacent to a box labeled $i$, we do nothing.
\end{defn}

For the definitions to follow, we will allow boxes to be labeled with the symbol $\circ$ in addition to labels in $\mathbb{N}$. 
For example,
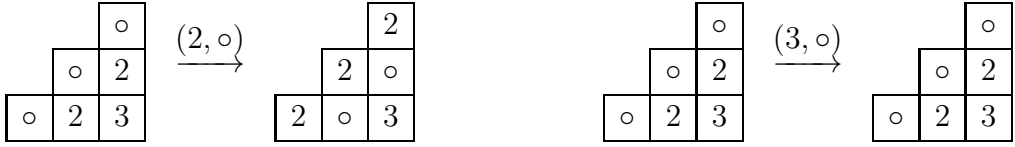
\[\begin{ytableau}
\none & \none & \circ \\
\none & \circ & 2 \\
\circ & 2 & 3
\end{ytableau} 
\quad\underrightarrow{(2,\circ)}\quad
\begin{ytableau}
\none & \none & 2 \\
\none & 2 & \circ \\
2 & \circ & 3
\end{ytableau}\hspace{1in}
\begin{ytableau}
\none & \none & \circ \\
\none & \circ & 2 \\
\circ & 2 & 3
\end{ytableau} 
\quad\underrightarrow{(3,\circ)}\quad
\begin{ytableau}
\none & \none & \circ \\
\none & \circ & 2 \\
\circ & 2 & 3
\end{ytableau} 
\]
The next step in defining shifted $K$-jdt is to define a shifted $K$-jdt slide.
Let $\Lambda$ denote the union of all shifted shapes $\lambda$, that is, $\Lambda=\cup \lambda$. 

\begin{defn}\label{def:slide}
Let $T$ be an increasing tableau of shifted skew shape $\lambda/\mu$ with values in the interval $[a,b]\subset \mathbb{N}$ and $C$ be a subset of the maximal boxes of $\mu$.
Label each box in C with $\circ$. 
The \textit{forward slide} $kjdt_{C}(T)$ of $T$ starting from $C$,  is 
\[
kjdt_{C}(T) = (b,\circ)(b-1,\circ)\ldots(a+1,\circ)(a,\circ)(T)
\]
Similarly, if $\widehat{C}$ is a subset of the minimal boxes of $\Lambda/\lambda$ labeled by $\circ$, the \textit{reverse slide} $\widehat{kjdt}_{\widehat{C}}(T)$ of $T$ starting from $\widehat{C}$ is
\[
\widehat{kjdt}_{\widehat{C}}(T) = (a,\circ)(a+1,\circ)\ldots(b-1,\circ)(b,\circ)(T).
\]
\end{defn}
For example, given $C = \{(1,2),(2,1)\}$ and $\hat{C} = \{(2,4),(3,3), (4,2), (4,1)\}$, we see forward and reverse slides of a tableau $T$ are
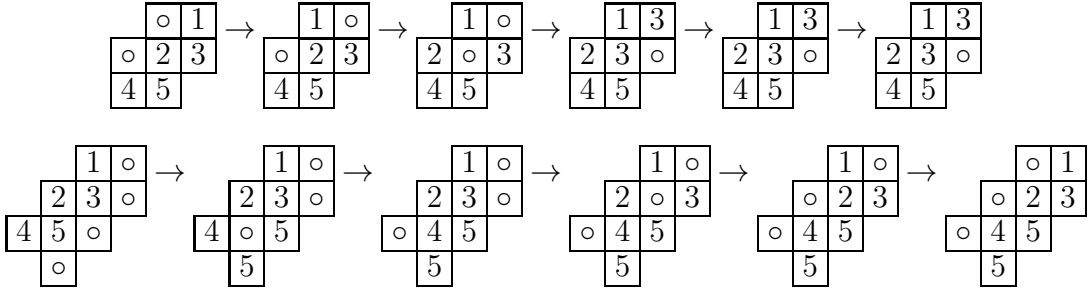
\[
\ytableausetup{boxsize=0.45cm}
\begin{ytableau}
 \none & \circ & 1  \\
 \circ  & 2 & 3 \\
 4 & 5 \\
\end{ytableau}\rightarrow
\begin{ytableau}
 \none & 1 & \circ  \\
 \circ  & 2 & 3 \\
 4 & 5 \\
\end{ytableau}\rightarrow
\begin{ytableau}
 \none & 1 & \circ  \\
 2 & \circ & 3 \\
 4 & 5 \\
\end{ytableau}
\rightarrow
\begin{ytableau}
 \none & 1 & 3 \\
 2 & 3 & \circ \\
 4 & 5 \\
\end{ytableau}\rightarrow
\begin{ytableau}
 \none & 1 & 3 \\
 2 & 3 & \circ \\
 4 & 5 \\
\end{ytableau}\rightarrow
\begin{ytableau}
 \none & 1 & 3 \\
 2 & 3 & \circ \\
 4 & 5 \\
\end{ytableau}
\]

\[
\begin{ytableau}
 \none & \none & 1 & \circ\\
\none &  2 & 3 & \circ\\
 4 & 5 & \circ\\
\none & \circ
\end{ytableau}\rightarrow
\begin{ytableau}
 \none & \none & 1 & \circ\\
\none &   2 & 3 & \circ\\
 4 & \circ & 5\\
 \none & 5
\end{ytableau}\rightarrow
\begin{ytableau}
 \none & \none & 1 & \circ\\
\none &  2 & 3 & \circ\\
 \circ & 4 & 5\\
 \none & 5
\end{ytableau}
\rightarrow
\begin{ytableau}
\none &  \none &1 & \circ\\
\none   & 2 & \circ & 3\\
 \circ & 4 & 5\\
 \none & 5
\end{ytableau}\rightarrow
\begin{ytableau}
\none  & \none & 1 & \circ\\
\none   & \circ & 2 & 3\\
 \circ & 4 & 5\\
\none & 5
\end{ytableau}\rightarrow
\begin{ytableau}
 \none & \none  & \circ & 1\\
  \none & \circ & 2 & 3\\
 \circ & 4 & 5\\
\none & 5
\end{ytableau}
\]
Observe that forward slides can transform a skew shape into a straight shape, while reverse slides can do the opposite. 

\begin{defn}
A \emph{shifted $K$-rectification} of an increasing shifted skew tableau $T$ is any shifted tableau that can be obtained from $T$ by a series of forward slides. 
\end{defn}

In~\cite{clifford2014k}, Clifford, Thomas, and Yong show that the shifted $K$-rectification of an increasing shifted tableau will be an increasing shifted tableau.
However, this tableau need not be unique; different choices of boxes labeled $\circ$ for each slide, called the \textit{rectification order}, may lead to different rectifications. 
For this reason, when studying shifted $K$-rectification, it is necessary to specify the rectification order. 
We will do so by marking boxes with an underlined version of a positive integer as in~\cite{ThomasYong}. 
The underlined boxes are labeled in reverse chronological order so that the subset $C$ for the last shifted $K$-jdt slide is indicated by boxes labeled $\underline{1}$, the second to last by $\underline{2}$, and so forth. 
See Example~\ref{ex:Krect} for an example of shifted $K$-rectification.

Now, we describe a particular rectification order for a shifted skew tableau $T$.
We will use it to establish the link between shifted Hecke insertion and the weak $K$-Knuth equivalence. 
As a preliminary step, we define the \textit{superstandard tableau} of shifted shape $\lambda=(\lambda_1,\lambda_2,\ldots,\lambda_k)$ to be the tableau with first row labeled by $1,2,\ldots,\lambda_1$, second row by $\lambda_1+1,\lambda_1+2,\ldots, \lambda_2$, and so on. 
For example,
\begin{center}
\begin{ytableau}
1 & 2 & 3 & 4 \\
\none & 5 & 6 \\
\none & \none & 7
\end{ytableau}
\end{center}
is the superstandard tableau of shape $(4,2,1)$.

For an increasing shifted skew tableau $T$ of shape $\lambda/\mu$, we define the \emph{superstandard rectification order} by filling the shape $\mu$ so the resulting tableau is superstandard with entries $[\underline{1},\underline{p}]$ where $p=|\mu|$.
The resulting tableau after applying shifted $K$-rectification is the \textit{superstandard $K$-rectification}. 
Superstandard rectification can be decomposed into the sequence of switch operations	
\[
(\underline{p},1),(\underline{p},2),\ldots,(\underline{p},q),(\underline{p-1},1),\ldots,(\underline{p-1},q),\ldots,(\underline{1},1),\ldots,(\underline{1},q)
\]
performed from left to right.
Here, $q$ is the largest entry of $T$.
This sequence of pairs is the \textit{standard switch sequence}.

\begin{ex}\label{ex:Krect} Given a skew increasing tableau
\ytableausetup{mathmode, boxsize=1.2em}
\begin{ytableau}
\none & \none & 1 \\
\none & 2 & 3 \\
1 & 4 \\
\end{ytableau}, we obtain the shifted superstandard $K$-rectification as follows:

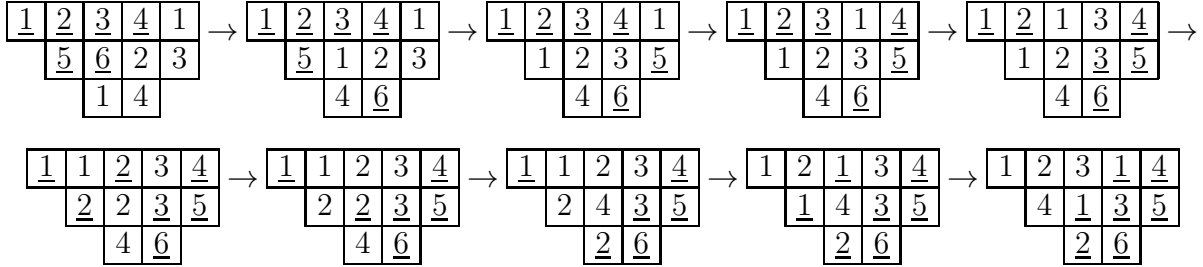
\[\begin{array}{l r}
\ytableausetup{mathmode, boxsize=1.2em}

\begin{ytableau}
\underline{1} & \underline{2} & \underline{3} & \underline{4} & 1 \\
\none & \underline{5} & \underline{6} & 2 & 3\\
\none & \none & 1 & 4\\
\end{ytableau} 

\rightarrow

\begin{ytableau}
\underline{1} & \underline{2} & \underline{3} & \underline{4} & 1 \\
\none & \underline{5} & 1 & 2 & 3\\
\none & \none & 4 & \underline{6}\\
\end{ytableau}

\rightarrow

\begin{ytableau}
\underline{1} & \underline{2} & \underline{3} & \underline{4} & 1 \\
\none & 1 & 2 & 3 & \underline{5}\\
\none & \none & 4 & \underline{6}\\
\end{ytableau}

\rightarrow

\begin{ytableau}
\underline{1} & \underline{2} & \underline{3} & 1 & \underline{4} \\
\none & 1 & 2 & 3 & \underline{5}\\
\none & \none & 4 & \underline{6}\\
\end{ytableau}

\rightarrow

\begin{ytableau}
\underline{1} & \underline{2} & 1 & 3 & \underline{4} \\
\none & 1 & 2 & \underline{3} & \underline{5}\\
\none & \none & 4 & \underline{6}\\
\end{ytableau} 

\rightarrow
\end{array}\]

\[\begin{array}{l r}
\ytableausetup{mathmode, boxsize=1.2em}

\begin{ytableau}
\underline{1} & 1 & \underline{2} & 3 & \underline{4} \\
\none & \underline{2} & 2 & \underline{3} & \underline{5}\\
\none & \none & 4 & \underline{6}\\
\end{ytableau} 

\rightarrow

\begin{ytableau}
\underline{1} & 1 & 2 & 3 & \underline{4} \\
\none & 2 & \underline{2} & \underline{3} & \underline{5}\\
\none & \none & 4 & \underline{6}\\
\end{ytableau}

\rightarrow

\begin{ytableau}
\underline{1} & 1 & 2 & 3 & \underline{4} \\
\none & 2 & 4 & \underline{3} & \underline{5}\\
\none & \none & \underline{2} & \underline{6}\\
\end{ytableau}

\rightarrow

\begin{ytableau}
1 & 2 & \underline{1} & 3 & \underline{4} \\
\none & \underline{1} & 4 & \underline{3} & \underline{5}\\
\none & \none & \underline{2} & \underline{6}\\
\end{ytableau}

\rightarrow

\begin{ytableau}
1 & 2 & 3 & \underline{1} & \underline{4} \\
\none & 4 & \underline{1} & \underline{3} & \underline{5}\\
\none & \none & \underline{2} & \underline{6}\\
\end{ytableau}
\end{array}\]
The sequence of nontrivial switch operators used in this shifted $K$-rectification is $$(\underline{6},1),(\underline{6},4),(\underline{5},1),(\underline{5},2), (\underline{5},3),(\underline{4},1),(\underline{3},1),(\underline{3},3),(\underline{2},1),(\underline{2},2),(\underline{2},4),(\underline{1},1),(\underline{1},2),(\underline{1},3),(\underline{1},4)$$
We see that the superstandard $K$-rectification of $T$ is
\begin{ytableau}
1 & 2 & 3 \\
\none & 4 \\
\end{ytableau}.
\end{ex}

Following Definition~\ref{def:slide}, we perform the switches in the order determined by the standard switch sequence. 
However, there is some flexibility in this ordering. 
\begin{lem}\emph{\cite[Lemma~4.4]{clifford2014k}}\label{switchcommute}
If $i\neq j$ and $r\neq s$ then the switch operators ($\underline{i}, r$) and ($\underline{j}, s$) commute. 
\end{lem}
This motivates the following definition.
\begin{defn}\cite[Section 4]{ThomasYong}
A \textit{viable switch sequence} is a sequence of switch operators, with the following properties:

\begin{enumerate}
\item every switch ($\underline{i}, j$) occurs exactly once, for $1 \leq i \leq p$ and $1 \leq j \leq q$; 
\item for any $1 \leq i \leq p$, the pairs ($\underline{i},1$), $\ldots$, ($\underline{i}, q$) occur in  that relative order;
\item for any $1 \leq j \leq q$, the pairs ($\underline{p},j$), $\ldots$, ($\underline{1}, j$) occur in  that relative order.\\
\end{enumerate}
\end{defn}

\begin{cor}
\label{switchsequence} 
Any viable \text{switch} sequence can be used to calculate shifted K-rectification. 
\end{cor}

\begin{proof} It is straightforward to show that any viable switch sequence can be obtained from the standard switch sequence by repeated applications of Lemma \ref{switchcommute}. 
\end{proof}

For any word $w=w_1\ldots w_n$, let $T_w$ denote the shifted skew tableau consisting of $n$ boxes on the antidiagonal with reading word $w$. 
The next theorem explains the relationship between shifted Hecke insertion and shifted $K$-jdt, affirming that shifted Hecke insertion is the correct $K$-theoretic analogue of Sagan-Worley insertion.

\begin{thm}
\label{jdtgivesinsertion} 
Let $w$ be a word and $P$ be the shifted superstandard $K$-rectification of the antidiagonal tableau of $T_w$.
Then $P = \Tab(w)$. 
\end{thm}

\begin{proof} We will closely follow the structure of the proof of \cite[Theorem 4.2]{ThomasYong}.  

We proceed by induction on $n$, the length of the word $w$. 
It is easy to check that the result is true for $n=1$ and $n=2$. 

Let $P$ be the shifted tableau resulting from superstandard $K$-rectification of $T_{w_1\ldots w_{n-1}}$. 
By induction, $P=\Tab(w_1\ldots w_{n-1})$. 
The non-underlined labels in the figure on the right indicate elements of $P*w_n$, and the underlined labels determine the steps remaining to finish the superstandard $K$-rectification of $T_w$. 

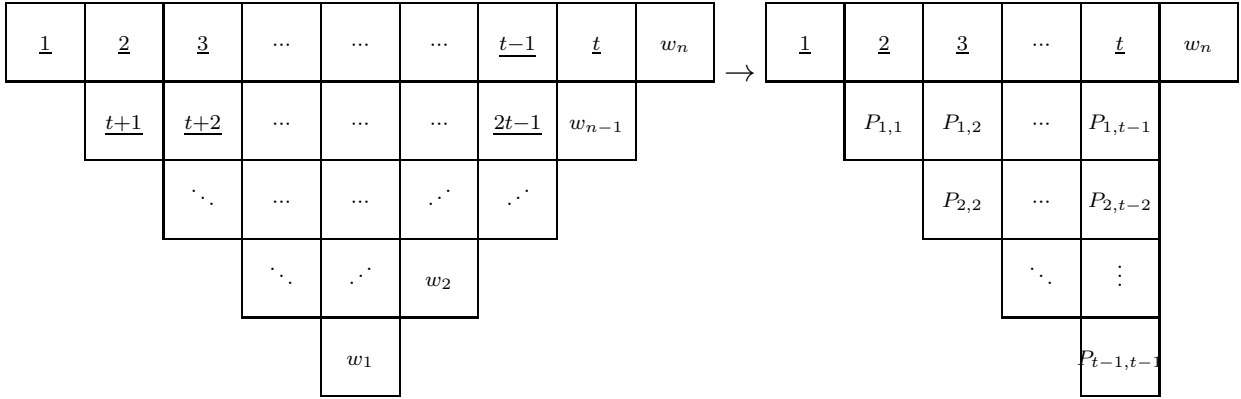
\begin{center}
\ytableausetup{mathmode, boxsize=2.5em}
\begin{ytableau}
\large
\sss \underline{1} &\sss \underline{2} &\sss \underline{3} &\sss \cdots &\sss \cdots &\sss \cdots &\sss \underline{t-1} &\sss \underline{t} &\sss w_n \\
\none &\sss \underline{t+1} &\sss \underline{t+2} &\sss \cdots &\sss \cdots & \sss\cdots &\sss \underline{2t-1} &\sss w_{n-1} \\
\none & \none & \sss\ddots &\sss \cdots &\sss \cdots &\sss \iddots &\sss \iddots\\
\none & \none & \none &\sss \ddots & \sss\iddots &\sss w_2\\
\none & \none & \none & \none &\sss w_1\\
\end{ytableau}
$\rightarrow$
\begin{ytableau}
\sss\underline{1} &\sss \underline{2} &\sss \underline{3} &\sss \cdots &\sss \underline{t} &\sss w_n \\
\none &\sss P_{1,1} &\sss P_{1,2} &\sss \cdots &\sss P_{1,t-1} \\
\none & \none & \sss P_{2,2} &\sss \cdots &\sss P_{2,t-2}\\
\none & \none & \none &\sss \ddots &\sss \vdots\\
\none & \none & \none & \none &\sss P_{t-1,t-1}\\
\end{ytableau}
\end{center}

\noindent
It remains to show that the shifted superstandard $K$-rectification on $P*w_n$ yields the same tableau as $P\insert w_n$. 

Note that any viable switch sequence for the underlined entries $\underline{1},\ldots, \underline{t}$, when concatenated with a viable switch sequence for the larger underlined entries, will give a viable switch sequence for $T_w$.
We construct a viable switch sequence for which shifted $K$-rectification coincides with the tableau obtained by insertion $P\insert w_n$. (See Figure~ 1 for an example.)
Our result then follows from Corollary \ref{switchsequence}.

For $i \geq 1$, let $y_i$ be the letter that is to be inserted into row (resp column) $i$ in the shifted Hecke insertion of $w_n$ into $P$, so $y_1=w_n$. 
Our strategy is to begin with an
%
$i$th row is of the form 
\[\ytableausetup
{mathmode, boxsize=2em}
\begin{ytableau}
\large
\underline{t_1} & \underline{t_2} & \cdots & \underline{t_k} & y_i & \underline{t_{k+1}} & \cdots & \underline{t_m}\\
\end{ytableau}\]
and, after several switch operations, construct a new $i$th row of the form  
\[\ytableausetup
{mathmode, boxsize=2em}
\begin{ytableau}
\large
\underline{t_1} & x_1 & \cdots & x_n & y_i & x_{n+1} & \cdots & x_{m-1} \\
\end{ytableau}\ .\]
Here if $x_j$ is underlined, then $x_k$ is underlined for all $k > j$.
In other words, we start with a row that has only one non-underlined entry. 
Such rows are said to be of \emph{normal form}.
We end with a row that begins with an underlined entry and may end with them.
Such rows are said to be \emph{cleared}.
For row bumps, our procedure of switches will turn the $i$th row, which is normal, into a cleared row.
The resulting $i+1$th row will again be normal or filled entirely with underlined entries.
The analogous sequence for column bumps will be a trivial modification (switch up instead of left, left instead of up).



Starting with tableau $P*w_n$, whose first row is of normal form, we now construct a viable switch sequence involving local applications of the cases below to get the superstandard $K$-rectification. 
Suppose after some sequence of switches, we have arrived at a tableau where row $i$ is of normal form and all higher rows are cleared. 
There are four cases for the row bumping step of shifted Hecke insertion.  
We will see the cases for column bumping are identical.

Let $T$ be a tableau with the first $i-1$ rows cleared and the $i$th row normal with entry $y_i$ not underlined.
We use $z_i$ to denote the smallest entry greater than $y_i$ in the $i+1$th row.
Let $k$ be the position of $z_i$ in the $i+1$th row, should $z_i$ exist.
If the $i+2$th row has a $k-1$th and $k$th entries, we denote them $a$ and $b$, respectively.

\noindent \textbf{Case 1:} $y_i$ is greater than or equal to any entry in the $i+1$th row.

Switch $y_i$ as far to the left as possible.
If it is equal to the last entry of the $i+1$th row, the final switch will merge $y_i$ with the last entry.
Then, from right to left, switch each entry in the $i+1$th row into the $i$th row.
The resulting $i$th row is cleared, coinciding with shifted Hecke insertion.
The $i+1$th row has only underlined entries.

\noindent \textbf{Case 2:} $z_i$ exists, and the entry preceding it is less than $y_i$.

Switch $y_i$ as far to the left as possible, so that it is above $z_i$.
Then, from right to left, switch each other entry in the $i+1$th row into the $i$th row:
\[\begin{array}{l c r}
\ytableausetup{mathmode, boxsize=2.5em}
\begin{ytableau}
\none[\cdots] & \underline{k-1} & y_i & \underline{k} & \none[\cdots] \\
\none[\cdots] & x_{k-1} & z_i & x_k & \none[\cdots]\\
\end{ytableau}

\rightarrow

\ytableausetup{mathmode, boxsize=2.5em}
\begin{ytableau}
\none[\cdots] & x_{k-1} & y_i & x_k & \none[\cdots] \\
\none[\cdots] & \underline{k-1} & z_i & \underline{k} & \none[\cdots]\\
\end{ytableau}
\end{array}\]
The resulting $i$th row is cleared, coinciding with shifted Hecke insertion.
The $i+1$th row is now normal.

\noindent \textbf{Case 3:} $z_i$ exists, the entry preceding it is equal to $y_i$, and $z_i < a$.

Switch $y_i$ to the left until it is above $z_i$.
The next switch merges $y_i$ with the $k-1$th entry in the $i+1$th row, immediately preceding $z_i$.
Since $z_i < a$, we may switch $z_i$ into both of the entries vacated by $y_i$.
\[\begin{array}{l c r}
\ytableausetup{mathmode, boxsize=1.5em}
\begin{ytableau}
\none[\cdots] & \underline{t} & y_i & \none[\cdots] \\
\none[\cdots] & y_i & z_i & \none[\cdots]\\
\none[\cdots] & a & b & \none[\cdots]\\
\end{ytableau}

\rightarrow

\ytableausetup{mathmode, boxsize=1.5em}
\begin{ytableau}
\none[\cdots] & y_i & \underline{t} & \none[\cdots] \\
\none[\cdots] & \underline{t} & z_i & \none[\cdots]\\
\none[\cdots] & a & b & \none[\cdots]\\
\end{ytableau}

\rightarrow

\ytableausetup{mathmode, boxsize=1.5em}
\begin{ytableau}
\none[\cdots] & y_i & z_i & \none[\cdots] \\
\none[\cdots] & z_i & \underline{t} & \none[\cdots]\\
\none[\cdots] & a & b & \none[\cdots]\\
\end{ytableau}
\end{array}\]
From right to left, switch the remaining entries in the $i+1$th row into the $i$th row.
The resulting $i$th row is cleared, coinciding with shifted Hecke insertion.
The $i+1$th row is now normal.

\noindent \textbf{`Case 4:'} $z_i$ exists, the entry preceding it is equal to $y_i$, and $z_i > a$.

Switch $y_i$ to the left until it is above $z_i$.
The next switch merges $y_i$ with the $k-1$th entry in the $i+1$th row, immediately preceding $z_i$.
Since $z_i >a$, we must switch $a$ into the $i+1$th row before switching $z_i$ into the $i$th row.
Then $b$ splits into two rows.
\[\begin{array}{l c r}
\ytableausetup{mathmode, boxsize=1.5em}
\begin{ytableau}
\none[\cdots] & \underline{t} & y_i & \none[\cdots] \\
\none[\cdots] & y_i & z_i & \none[\cdots]\\
\none[\cdots] & a & b & \none[\cdots]\\
\end{ytableau}

\rightarrow

\ytableausetup{mathmode, boxsize=1.5em}
\begin{ytableau}
\none[\cdots] & y_i & \underline{t} & \none[\cdots] \\
\none[\cdots] & \underline{t} & z_i & \none[\cdots]\\
\none[\cdots] & a & b & \none[\cdots]\\
\end{ytableau}

\rightarrow

\ytableausetup{mathmode, boxsize=1.5em}
\begin{ytableau}
\none[\cdots] & y_i & \underline{t} & \none[\cdots] \\
\none[\cdots] & a & z_i & \none[\cdots]\\
\none[\cdots] & \underline{t} & b & \none[\cdots]\\
\end{ytableau}

\rightarrow

\ytableausetup{mathmode, boxsize=1.5em}
\begin{ytableau}
\none[\cdots] & y_i & z_i & \none[\cdots] \\
\none[\cdots] & a & \underline{t} & \none[\cdots]\\
\none[\cdots] & \underline{t} & b & \none[\cdots]\\
\end{ytableau}

\rightarrow

\ytableausetup{mathmode, boxsize=1.5em}
\begin{ytableau}
\none[\cdots] & y_i & z_i & \none[\cdots] \\
\none[\cdots] & a & b & \none[\cdots]\\
\none[\cdots] & b & \underline{t} & \none[\cdots]\\
\end{ytableau}
\end{array}\]
Then from right to left, switch the remaining entries in the $i+1$th row into the $i$th row.
Next, switch the remaining entries in the $i+2$th row into the $i+1$th row.
Note the $i$th and $i+1$th rows are now cleared.
To see they coincide with shifted Hecke insertion, observe that in this case $z_i$ fails to insert into the $i+2$th row since the resulting tableau would not be increasing.
However, it still bumps $b$ (should it exist), which is now the sole entry in the $i+2$th row that is not underlined.
Therefore, the $i{+}2$th row is normal.

In fact, this case is slightly more complicated.
Let $a_1, a_2, \dots, a_p$ and $b_1,b_2,\dots, b_p$ be the entries in the $k-1$th and $k$th columns in rows $i+2, i+2, \dots, i+p-1$.
If $a_i < b_{i+1}$ for each $i$, then the above case applies for each of the rows $i+2, i+2, \dots, i+p-1$.
We are then performing $p$ bumps simultaneously, resulting in $p$ new cleared rows with the $i+p+1$th row normal with $b_p$ the sole entry that is not underlined. 
The previous argument is the case where $p = 1$.

If column insertion is never triggered, every row of the resulting tableau will be cleared.
Then row by row from bottom to top, use switches to shift each row over, filling the shifted shape.

We will now describe how to transition to column insertion.
The bumping procedure will then be identical to the row bumping, except initial entries will not be underlined.
 
When $y_i$ bumps an element $z_i$ from column $i+1$ of the main diagonal, we begin column insertion.
Note that the first $i-1$ rows are cleared, that is of the form
\[\ytableausetup
{mathmode, boxsize=2em}
\begin{ytableau}
\large
\underline{t_1} & x_1 & \cdots & x_n & y_i & x_{n+1} & \cdots & x_{m-1} \\
\end{ytableau}\ .\]
In these rows, from top to bottom, shift all entries in the first $i+1$ columns to the left, filling the initial underlined entry.
As a result, the $i+1$th column will be normal with $z_i$ the sole entry that is not underlined.
The first $i$ columns will be the same as found in shifted Hecke insertion, as will columns $i+2$ and on.
We have now transitioned to column insertion, inserting $z_i$ into the $i+2$th column.

Since each sequence of switches described is viable, the whole sequence described is viable.
Moreover, it reproduces shifted Hecke insertion at every step.
This completes our proof.
\end{proof}

\begin{figure}\label{fig:hv}
\caption{$T \leftarrow 5 = T$ using shifted jeu de taquin.}

\[ T = 
\ytableausetup{boxsize=1.2em}
\begin{ytableau}
1 & 2 & 3 & 4 & 6 \\
\none & 4 & 5 & 6 & 8 \\
\none & \none & 6 & 7
\end{ytableau}
\]

\begin{enumerate}

\item We begin with $T*5$:
\[\begin{ytableau}
\underline{1} & \underline{2} & \underline{3} & \underline{4} & \underline{5} & \underline{6} & 5 \\
\none & 1 & 2 & 3 & 4 & 6 \\
\none & \none & 4 & 5 & 6 & 8 \\
\none & \none & \none & 6 & 7
\end{ytableau}
\]

\item We switch 5 as far as possible to the left using switch sequence $(\underline{6},5)$:
\[\begin{ytableau}
\underline{1} & \underline{2} & \underline{3} & \underline{4} & \underline{5} & 5 & \underline{6}  \\
\none & 1 & 2 & 3 & 4 & 6 \\
\none & \none & 4 & 5 & 6 & 8 \\
\none & \none & \none & 6 & 7
\end{ytableau}
\]

\item Case 2:  Using the switch sequence $(\underline{5}, 4),(\underline{4},3),(\underline{3}, 2),(\underline{2},1)$, we obtain
\[\begin{ytableau}
\underline{1} & 1 & 2 & 3 & 4 & 5 & \underline{6}  \\
\none & \underline{2} & \underline{3} & \underline{4} & \underline{5} & 6 \\
\none & \none & 4 & 5 & 6 & 8 \\
\none & \none & \none & 6 & 7
\end{ytableau}
\]

\item Case 4:  A vertical violation is about to occur, so we perform the switch sequence $(\underline{5},6),(\underline{4},5),(\underline{3},4)$ to obtain
\[\begin{ytableau}
\underline{1} & 1 & 2 & 3 & 4 & 5 & \underline{6}  \\
\none & \underline{2} & 4 & 5 & 6 & \underline{5} \\
\none & \none & \underline{3} & \underline{4} & \underline{5} & 8 \\
\none & \none & \none & 6 & 7
\end{ytableau}
\]

\noindent Then we perform the switch sequence $(\underline{5},7),(\underline{5},8),(\underline{4},6)$ to obtain
\[\begin{ytableau}
\underline{1} & 1 & 2 & 3 & 4 & 5 & \underline{6}  \\
\none & \underline{2} & 4 & 5 & 6 & 8 \\
\none & \none & \underline{3} & 6 & 7 & \underline{5} \\
\none & \none & \none & \underline{4} & \underline{5}
\end{ytableau}
\]

\item Terminal step: We switch the underlined entries on the main diagonal to the outside from bottom to top, using switch sequence $(\underline{3},6),(\underline{3},7),(\underline{2},4),(\underline{2},5),(\underline{2},6),(\underline{2},8),(\underline{1},1),$ $(\underline{1},2),(\underline{1},3),(\underline{1},4),(\underline{1},5)$:
\[\begin{ytableau}
 1 & 2 & 3 & 4 & 5 & \underline{1} & \underline{6}  \\
\none & 4 & 5 & 6 & 8 &\underline{2} \\
\none & \none & 6 & 7 & \underline{3} & \underline{5} \\
\none & \none & \none & \underline{4} & \underline{5}
\end{ytableau}
\]
which is the same as the result of shifted Hecke insertion.  
The switch sequence used is valid.  \\
\end{enumerate}
\end{figure} 

\ 

Using a result of Buch and Samuel~\cite{BuchandSamuel2013}, we can now relate shifted Hecke insertion to the weak $K$-Knuth relations.    
Tableaux $T$ and $T'$ are called \textit{jeu de taquin equivalent} if one can be obtained from another by a sequence of jeu de taquin slides.

\begin{thm}\cite[Theorem~7.8]{BuchandSamuel2013}
Let $T$ and $T\pr$ be increasing shifted tableaux. 
Then 
\[
\row(T)\ \ksw\ \row (T\pr)
\] 
if and only if $T$ and $T\pr$ are jeu de taquin equivalent.  
\end{thm}

The previous theorem says that weak $K$-Knuth equivalence of tableaux is the same as jeu de taquin equivalence of tableaux. From this point on, we refer to both as ``equivalence.'' 

\begin{cor}\label{insertionimplieskknuth}
If $\Tab(u) = \Tab(v)$, then $u\  \ksw \ v$.  
\end{cor}

\begin{remark}\label{notURT}
The converse of this statement does not hold.  
Consider the words 12453 and 124533, which are easily seen to be weakly $K$-Knuth equivalent.  
We compute that shifted Hecke insertion gives the following distinct tableaux.
\[ \begin{array}{l c r}
\ytableausetup{boxsize=1.5em}
\Tab(12453) = 
\begin{ytableau}
1 & 2 & 3 & 5 \\
\none & 4
\end{ytableau}
& &
\Tab(124533) = 
\begin{ytableau}
1 & 2 & 3 & 5 \\
\none & 4 & 5
\end{ytableau}
\end{array}\]
\end{remark}

\subsection{Unique Rectification Targets}
As we have seen in Remark \ref{notURT}, weak $K$-Knuth equivalence classes may have several corresponding insertion tableaux. 
This is a key difference between weak $K$-Knuth equivalence and the classical Knuth equivalence. 
Of particular importance in our setting are the classes of words with only one tableau. 
Such classes of words will be crucial in our Littlewood-Richardson rule.

\begin{defn}\cite[Definition 3.5]{BuchandSamuel2013}
\label{def:urt}
An increasing shifted tableau $T$ is a \textit{unique rectification target}, or a URT, if it is the only tableau in its weak $K$-Knuth equivalence class.
Equivalently, $T$ is a URT if for every $w\ \ksw\ \row(T)$ we have $\Tab(w) = T$.  
If $\Tab(w)$ is a URT, we call the equivalence class of $w$ a \textit{unique rectification class}.  
\end{defn}

We refer the reader to \cite{BuchandSamuel2013, clifford2014k} for a more detailed discussion of URTs for shifted tableaux and straight shape tableaux. 
The tableaux given in Example~\ref{notURT} are equivalent to each other, and hence neither is a URT.  

The \emph{minimal increasing shifted tableau} $M_\lambda$ of a shifted shape $\lambda$ is the tableau obtained by filling the boxes of $\lambda$ with the smallest values allowed in an increasing tableau. 
For example,
\[\begin{array}{l r}
M_{(4,2)} = 
\begin{ytableau}
1 & 2 & 3 & 4 \\
\none & 3 & 4
\end{ytableau} &
M_{(5, 2, 1)} = 
\begin{ytableau}
1 & 2 & 3 & 4 & 5 \\
\none & 3 & 4\\
\none & \none & 5
\end{ytableau} 
\end{array}\]
are minimal increasing tableaux.

\begin{thm}\label{thm:urt} \ 

\begin{enumerate}
\item \emph{\cite[Corollary 7.2]{BuchandSamuel2013}} Minimal increasing shifted tableaux are URTs.

\item \emph{\cite[Theorem 1.1]{clifford2014k}} Superstandard shifted tableaux are URTs.
\end{enumerate}

\end{thm}
As a consequence, we see there are URTs for every shifted shape.


\section{Shifted weak stable Grothendieck polynomials}
\subsection{Weak shifted stable Grothendieck polynomials}
We define the weak shifted stable Grothendieck polynomial $K_\lambda$ as a weighted generating function over weak set-valued shifted tableaux.

\begin{defn}\label{wsvst}
A \textit{weak set-valued shifted tableau} is a filling of the boxes of a shifted shape with finite, nonempty multisets of primed and unprimed positive integers with ordering $1\pr<1<2\pr<2<\cdots $ such that:

\begin{enumerate}
\item The smallest number in each box is greater than or equal to the largest number in the box directly to the left of it, if that box exists.
\item The smallest number in each box is greater than or equal to the largest number in the box directly above it, if that box exists.
\item There are no primed entries on the main diagonal.
\item Each unprimed integer appears in at most one box in each column.
\item Each primed integer appears in at most one box in each row. \\
\end{enumerate} 
\end{defn}
Given any weak set-valued shifted tableau $T$, we define $x^T$ to be the monomial $\prod_{i\geq 1}x_i^{a_i}$, where $a_i$ is the number of occurrences of $i$ and $i\pr$ in $T$.
For example, the weak set-valued tableaux $T_1$ and $T_2$ below have $x^{T_1} = x_1x_2x_3x_4x_5x_6x_7$ and $x^{T_2} = x_1x_2x_3^2x_4x_5x_6^2$.
\begin{center}
$T_1 =$ 
\ytableausetup{boxsize=.72cm}
\begin{ytableau}
1 & 2 & 3\pr 4\\
\none & 5 & 6 7\pr
\end{ytableau} \hspace{1in}
$T_2 =$ 
\begin{ytableau}
1 & 2 & 3\pr 3\\
\none & 4 & 5 6\pr 6
\end{ytableau}
\end{center} 

Recall that we denote a shifted shape $\lambda$ as $(\lambda_1, \lambda_2, \cdots \lambda_l)$, where $l$ is the number of rows and $\lambda_i$ is the number of boxes of $i$th row.

\begin{defn}\label{k-lambda defn}
The \textit{weak shifted stable Grothendieck polynomial} is  
\[
K_\lambda = \sum_{T}x^T,
\] 
where the sum is over the set of weak set-valued tableaux $T$ of shape $\lambda$. 
\end{defn}

\begin{ex}We have 
\[K_{(2,1)} = x_1^2x_2+2x_1x_2x_3+3x_1^2x_2^2+5x_1^2x_2x_3+5x_1x_2^2x_3+\cdots,\] where the coefficient of $x_1^2x_2^2$ is 3 because of the tableaux shown below. \begin{center}
\begin{ytableau}
1 1 & 2\pr\\
\none & 2
\end{ytableau} \hspace{.8in}
\begin{ytableau}
1 & 1 2\pr\\
\none & 2
\end{ytableau} \hspace{.8in}
\begin{ytableau}
1 & 1\\
\none & 22
\end{ytableau}
\end{center}
Note that the lowest degree terms of $K_\lambda$ are a sum over shifted semistandard Young tableaux, so they form the Schur $P$-function $P_\lambda$.
\end{ex}

We refer to the ring spanned by the odd power sum symmetric functions $\{p_1,p_3,p_3,\ldots\}$ as the ring of \textit{signed symmetric functions}. 
The Schur $P$-functions are a basis for this ring, so one might expect the $K_\lambda$ to lie in this ring. 
However, this is not the case, since the degree two terms of $K_{(1)}$ are not a multiple of $P_{(2)}$, which is the only Schur $P$-function with terms of degree two.
\begin{prop} 
$\{K_\lambda\}$ does not lie in the ring of signed symmetric functions.
\end{prop}

The weak shifted stable Grothendieck polynomial $K_\lambda$ is closely related to the shifted stable Grothendieck polynomial $GP_\lambda$ defined by Ikeda and Naruse in  \cite{ikeda2013k}. 
Here, 
\[
GP_\lambda = \sum_T (-1)^{|T|-|\lambda|} x^T,
\]
where we sum over all set-valued shifted tableaux $T$ of shape $\lambda$ (see Definition~\ref{def:setvaluedshifted}), $|T|$ is the degree of $x^T$, and $|\lambda|$ is the number of boxes in $\lambda$.

To demonstrate this relationship, we require an observation.
For each set-valued shifted tableau $T$, we define a family $\mathcal{T}$ of weak set-valued shifted tableaux of the same shape, where $W\in\mathcal{T}$ if and only if $W$ can be obtained from $T$ by changing subsets in boxes of $T$ into multisets. 
Conversely, given a weak set-valued shifted tabelau $W$, we can identify the proper set-valued shifted tableau $T$ by turning multisets into subsets containing the same primed and unprimed integers. 
For example, given $T$ below, $W_1$ and $W_2$ are in $\mathcal{T}$.
\begin{center}
$T=$
\ytableausetup{boxsize=1cm}
\begin{ytableau}
12 & 3' 3 & 4 5' \\
\none & 4
\end{ytableau}\hspace{.5in}
$W_1=$
\begin{ytableau}
1122 & 3' 3 & 445' \\
\none & 444
\end{ytableau}\hspace{.5in}
$W_2=$
\begin{ytableau}
1112 & 3'33 & 4 4 5' 5' \\
\none & 44
\end{ytableau}
\end{center}

\begin{prop}\label{prop:Ikeda} We have
$$K_\lambda(x_1,x_2,\ldots)=(-1)^{|\lambda|} GP_\lambda \left (\frac{-x_1}{1-x_1},\frac{-x_2}{1-x_2},\ldots \right).$$
\end{prop}

\begin{proof} 

For $T$ a set-valued shifted tableau with $x^T=x_1^{a_1}x_2^{a_2}x_3^{a_3}\ldots$,
\[
\displaystyle\sum_{W\in\mathcal{T}} x^W
= \left(\frac{x_1}{1-x_1}\right)^{a_1}\left(\frac{x_2}{1-x_2}\right)^{a_2}\left(\frac{x_3}{1-x_3}\right)^{a_3}\ldots.
\]
Therefore 
\begin{eqnarray*}
(-1)^{|\lambda|}GP_\lambda\left (\frac{-x_1}{1-x_1},\frac{-x_2}{1-x_2},\ldots \right) &=& (-1)^{|\lambda|}\displaystyle\sum_T (-1)^{|T|-|\lambda|}\left(\frac{-x}{1-x}\right)^T \\ 
&=& \displaystyle\sum_T \left(\frac{x}{1-x}\right)^T = \displaystyle\sum_T \displaystyle\sum_{W \in \mathcal{T}} x^W\\ 
&=& K_\lambda.
\end{eqnarray*}
\end{proof}

\begin{remark}
Our weak shifted stable Grothendieck polynomials $K_\lambda$ are related to the $GP_\lambda$ in the same way that the weak stable Grothendieck polynomials $J_\lambda$ are related to the stable Grothendieck polynomials $G_\lambda$ in \cite{lam2007combinatorial}. See \cite{PPK} for the analogue of Proposition \ref{prop:Ikeda}.
\end{remark}

\subsection{$K_\lambda$ and fundamental quasisymmetric functions}
The \emph{descent set} of a word $w = w_1w_2\ldots w_n$ is $\D(w) = \{i : w_i > w_{i+1}\}$. 
Similarly, the \emph{descent set} of a standard set-valued shifted tableau $T$ is
\[\D(T) =  
\left\lbrace \begin{array}{c c l}
\multirow{5}{*}{\textit{i}} &
\multirow{5}{*}{:} &
\text{both }i\text{ and }(i+1)\pr\text{ appear}\\
& & \qquad \textsc{ or } \\
& & i \text{ is strictly above } i+1 \\ 
& & \qquad \textsc{ or } \\
& &i\pr \text{ is weakly below } (i+1)\pr \text{ but not in the same box}
\end{array} \right\rbrace. \] 
To any $\D \subset [n-1]$, we associate the \emph{fundamental quasisymmetric function} 
\[
f_\D = \sum_{\substack{i_1\leq i_2\leq\ldots\leq i_n\\
i_j<i_{j+1}\text{ if } j\in \D(\alpha)}} x_{i_1}x_{i_1}\ldots x_{i_n}.
\]

For example, the word $w=354211$ has descent set $\D(w) = \{2,3,4\}$.
Its shifted Hecke insertion recording tableau is 
\begin{center}
$\rTab(w) = $
\ytableausetup{boxsize=.6cm}
\begin{ytableau}
1 & 2 & 4\pr & 5\pr6\pr\\
\none & 3 \\
\end{ytableau}. 
\end{center}
We see $\D(\rTab(w)) =\{2,3,4\}$ because 2 is strictly above 3, both 3 and $4\pr$ appear, and $4\pr$ is weakly below $5\pr$ but not in the same box. 
The associated fundamental quasisymmetric function is
\begin{align*}
f_{\{2,3,4\}} = &\  x_1^2 x_2 x_3 x_4^2 + x_1^2 x_2 x_3 x_5^2 + \dots + x_2^2 x_5 x_7 x_9^2 + \dots \\
& + x_1^2 x_2 x_3 x_4 x_5 + x_1^2 x_2 x_3 x_4 x_6 + \dots + x_3^2 x_5 x_6 x_8 x_9 + \dots \\
&+ x_1 x_2 x_3 x_4 x_5 x_6 + \dots + x_2 x_3 x_5 x_6 x_9 x_{10} + \dots.
\end{align*}
Note  $\D(\rTab(w))=\D(w)$. 
This is true in general. 

\begin{thm}
\label{thm:descent}
For any word $w=w_1w_2\ldots w_n$,
\[
\D(w) = \D(\rTab(w)).
\]
\end{thm}

\begin{proof}

Let $t_i$ denote the box where the insertion of $w_i$ terminates, and let $\rTab(t_i)$ denote the labels of box $t_i$ in $\rTab(w)$. 
Then $i \in \rTab(t_i)$ if insertion of $w_i$ ends in row insertion and $i\pr \in \rTab(t_i)$ if insertion of $w_i$ ends in column insertion or failed row insertion into an empty row.

First suppose $i\in \D(w)$, so $w_i > w_{i+1}$. 
Observe that boxes bumped by row insertion caused by $w_i$ will be weakly right of those bumped by row insertion caused by $w_{i+1}$. 
This means that if $w_i$ triggers column insertion, $w_{i+1}$ will trigger column insertion in a row weakly above the row where $w_{i}$ did. 
(If there is a failed row insertion into an empty row, we will say that column insertion was triggered in the last row.)
We consider four cases.

\noindent \textbf{Case 1:} $\rTab(t_{i}) = i\pr$ and $\rTab(t_{i+1}) = i+1$: 

This cannot occur since $w_{i+1}$ triggers row insertion in a row weakly above the row where $w_i$ triggers column insertion.

\noindent \textbf{Case 2:} $\rTab(t_i)=i$ and $\rTab(t_{i+1})=i+1$:
 
In this case, both the insertion of $w_i$ and of $w_{i+1}$ only use row insertion and never fail to insert into an empty row. 
Therefore, using the standard doubling argument for shifted tableaux (Section 7 of \cite{BuchandSamuel2013}), $t_i$ is strictly above $t_{i+1}$ by the corresponding result for Hecke insertion found in Lemma 2 of \cite{HeckeInsertion}. 
Hence $i\in \D(\rTab(w))$.
   
\noindent \textbf{Case 3:} $\rTab(t_i) = i$ and $\rTab(t_{i+1}) = (i+1)\pr$: 

Here, $i\in\D(\rTab(w))$ by definition.

\noindent \textbf{Case 4:} $\rTab(t_i) = i\pr$ and $\rTab(t_{i+1}) = (i+1)\pr$:

If $t_i=t_{i+1}$, then reverse insertion (see~\cite{PPshifted}) shows that $w_i \leq w_{i+1}$, which is a contradiction. 
This is because at each step in reverse insertion, a label is exchanged for a label that is weakly greater. 
We conclude that $i\pr$ and $(i+1)\pr$ will not be in the same box.

Suppose neither insertion ends in failed row insertion into an empty row. 
Since column bumping paths move weakly upward and $w_{i+1}$ triggers column insertion weakly above $w_i$, $w_{i+1}$ will cause a column bump in every column that $w_i$ does. 
Therefore $t_{i+1}$ will be weakly right of $t_i$ and weakly above $t_i$. 
We conclude that $i'$ is weakly below $(i+1)\pr$ but not in the same box, as desired.
 
If insertion of $w_i$ ends in a failed row insertion into an empty row, then $t_i$ is the box at the right of the last row of $\rTab(w_1\ldots w_i)$. 
Since $w_{i+1}$ triggers column insertion weakly above $w_i$, column bumping paths move weakly upward, and $t_i\neq t_{i+1}$, it follows that $t_{i+1}$ is weakly above $t_i$. 
 
If insertion of $w_{i+1}$ ends in failed row insertion into an empty row, then $w_{i+1}$ triggers column insertion in the last row, which means $w_i$ also triggers column insertion in the last row. 
It follows that $t_i=t_{i+1}$, which cannot occur.
Therefore $\D(w) \subseteq \D(\rTab(w))$. 

Next assume $i \not\in \D(w)$ so $w_i \leq w_{i+1}$. 
We consider the same four cases. 

\noindent \textbf{Case 1:} $\rTab(t_i)=i$ and $\rTab(t_{i+1})=i+1\pr$: 

Insertion of $w_{i+1}$ causes row bumping weakly to the right of that caused by $w_i$. 
It follows that in the row where $w_{i+1}$ begins column insertion, $w_{i}$ can only row bump the diagonal element and would begin column insertion. 
Therefore this case never occurs.

\noindent \textbf{Case 2:} $\rTab(t_i) = i$ and $\rTab(t_{i+1}) = i+1$: 

Using the standard doubling argument for shifted tableaux (see Section 7 in \cite{BuchandSamuel2013}), this implies that $t_i$ is weakly below $t_{i+1}$ by Lemma 2 in \cite{HeckeInsertion}. 

\noindent \textbf{Case 3:} $\rTab(t_i) = i\pr$ while $\rTab(t_{i+1}) = i+1$: 

In this case, $i \not\in \D(\rTab(w))$ by definition.

\noindent \textbf{Case 4:} $\rTab(t_i) = i\pr$ and $\rTab(t_{i+1}) = (i+1)\pr$: 

Here, both insertions trigger column insertion.     
Suppose neither insertion of $w_i$ nor insertion of $w_{i+1}$ ends in failed row insertion into an empty row. 
Insertion of $w_{i+1}$ will cause column bumping weakly lower than that of $w_i$ and so $t_{i+1}$ must be weakly to the left of $t_i$.
If $t_i\neq t_{i+1}$, then $t_{i+1}$ must be strictly below $t_i$ to maintain strictly increasing columns and weakly increasing rows in $\rTab(w)$.

If insertion of $w_i$ ends in failed row insertion into an empty row, then insertion of $w_{i+1}$ must also. 
This implies $t_i=t_{i+1}$.

If insertion of $w_{i+1}$ ends in failed row insertion into an empty row, then $t_{i+1}$ is in bottom row of $\rTab(w_1\ldots w_{i+1})$, and so either $t_i$ is strictly above $t_{i+1}$ or $t_i=t_{i+1}$.
Therefore, when $w_i \leq w_{i+1}$, we have $i \notin \D(\rTab(w))$.

We conclude that $\D(w) = \D(\rTab(w))$.  
\end{proof}

The remainder of this section is devoted to expressing the $K_\lambda$ as a sum of fundamental quasisymmetric functions. 
We say a monomial $\sigma = x_{s_1}x_{s_2}\ldots x_{s_r}$ \textit{agrees with} $\D \subset [n-1]$ if $s_i \leq s_{i+1}$ with strict inequality when $i \in \D$.
For a standard set-valued shifted tableau $T$, we \textit{relabel  $T$ by $\sigma$} to obtain $T(\sigma)$ by replacing the $i$th smallest letter in $T$ with $s_i$.
Here, $s_i$ is primed if $i$ was.  
For example, given $\sigma = x_1x_3x_5^2x_6x_7$, we have $T$ and $T(\sigma)$ below.
\ytableausetup{boxsize=.7cm}
\[T = \begin{ytableau}
1 2 & 3\pr 4& 6 \\
\none & 5
\end{ytableau}
\hspace{1in}
T(\sigma) = \begin{ytableau}
1 3 & 5\pr 5& 7 \\
\none & 6
\end{ytableau}
\]

\begin{lem}\label{otherdirection}
Let $T$ be a standard set-valued shifted tableau and let $\sigma = x_{s_1}x_{s_2}\ldots x_{s_r}$ agree with $\D(T)$.  
Then $T(\sigma)$ is a weak set-valued shifted tableau.  
\end{lem}
\begin{proof}
We verify that $T(\sigma)$ satisfied the five conditions in Definition~\ref{wsvst} for weak set-valued shifted tableaux.

First, observe that we had a strictly increasing tableau before and we replaced the alphabet with one that is weakly increasing.  Thus both the rows and columns will be weakly increasing, so $T(\sigma)$ satisfies conditions (1) and (2).  

Since there are no primed entries on the main diagonal of $T$ by definition, there will be no primed entries on the main diagonal of $T(\sigma)$.
Therefore, $T(\sigma)$ satisfied condition (3).  

Note that if $i$ appears in the same column as $i{+}1$ but not in the same box, then $i \in \D(T)$.
Therefore, $s_i < s_{i+1}$, so $T(\sigma)$ contains the unprimed value $s_i$ in at most one box of each column.
Similarly, if $i\pr$ appears in the same row as $i{+}1\pr$ but not in the same box, then $i \in \D(T)$.
Therefore, $s_i < s_{i+1}$, so $T(\sigma)$ contains the primed value $s_i$ in at most one box of each row.
We conclude that $T(\sigma)$ satisfies conditions (4) and (5), and hence is a shifted weak set-valued tableau.  
\end{proof}

We define the \textit{standardization} of a weak set-valued tableau $T$, denoted $\st(T)$, to be the refinement of the order of entries of $T$ given by reading each occurrence of $k$ in $T$ from left to right and each occurrence of $k\pr$ in $T$ from top to bottom, using the total order $(1\pr<1<2\pr<2<\cdots )$.  Notice that $\st(T)$ will be a standard set-valued shifted tableau.
For example, we have
\[\begin{array}{l r}
T = 
\ytableausetup{mathmode, boxsize=2.3em}
\begin{ytableau}
\sss 1 & \sss 2 \ 3' &\sss  4 \ 5 &\sss 6'\ 7'\ 7 &\sss 9 \ 10 \\
\none & \sss 4 \ 4 &\sss  6 \ 7 &\sss  8'\\
\none & \none &\sss 8\\
\end{ytableau}&
\quad 
\st(T) = 
\ytableausetup{mathmode, boxsize=2.3em}
\begin{ytableau}
\sss 1 &\sss 2 \ 3' &\sss 6 \ 7 &\sss 8' 10' 12 &\sss 15 \ 16 \\
\none &\sss 4 \ 5 &\sss 9 \ 11 &\sss 13'\\
\none & \none &\sss 14\\
\end{ytableau} 
\end{array}\]

We can now write $K_\lambda$ as a sum of fundamental quasisymmetric functions.
\begin{thm}\label{kisgood}
For any fixed increasing shifted tableau $T$ of shape $\lambda$,
\[K_\lambda = \sum_{\Tab(w) = T}f_{\D(w)}.\]
\end{thm}
\begin{proof}
Let $W$ be a weak set-valued shifted tableau of shape $\lambda$ with entries $s_1 \leq s_2 \leq \ldots \leq s_n$, some of which may be primed.
By Theorem~\ref{wordbijection}, there is a unique word $w$ such that $\Tab(w) = T$ and  $\rTab(w) = \st(W)$. 
We will show that $x^W$ agrees with $\D(w) = \D(\st(W))$.
By Lemma~\ref{otherdirection}, every $\sigma$ that agrees with $\D(w)$ corresponds to some $W$, so this would complete our proof.

Let $j$ be strictly above $j+1$ in $\st(W)$. 
If $s_j = s_{j+1}$, by the weakly increasing property, we see they must be in the same column.
This violates condition (4) in the definition of weak set-valued tableaux, so $s_j < s_{j+1}$.

Next assume $j$ and $(j+1)\pr$ both appear in $\st(W)$. 
By the definition of standardization, $s_j$ and $s_{j+1}$ cannot be the same number, as primed entries of the same value are standardized first.
Therefore, $\sigma_j<\sigma_{j+1}$.

Finally, let $(j+1)\pr$ be weakly above $j\pr$ and not in the same box.
If $s_j=s_{j+1}$, they cannot appear on the same row of $W$, as this would violate condition (5) in the definition of weak set-valued shifted tableaux. 
Moreover, if $s_{j+1}\pr$ were strictly above $s_j\pr$, it would have to be strictly smaller after standardization, which cannot occur.
Therefore, $s_j < s_{j+1}$, which completes our proof.

\end{proof}

\subsection{The symmetry of $K_\lambda$}
From Proposition~\ref{prop:Ikeda} and~\cite[Theorem 9.1]{ikeda2013k}, we can conclude that the weak shifted stable Grothendieck polynomial $K_\lambda$ is symmetric.
However, using shifted Hecke insertion, we provide a direct proof.
This also shows that $K_\lambda$ is a sum of stable Grothendieck polynomials.
Moreover, by Proposition~\ref{prop:Ikeda}, we obtain a new proof of symmetry for the $GP_\lambda$.
To our mind, this proof of symmetry is easier than the Ikeda-Naruse proof, as might be expected since their result follows as a consequence of symmetry for the more general factorial $K$-theoretic Schur $P$-functions.

A \emph{set-valued tableau} $T$ of ordinary (unshifted) shape satisfies conditions (1), (2) and (4) of Definition~\ref{def:setvaluedshifted} and has no primed entries.
We use these tableaux to define a signless version of the stable Grothendieck polynomials---which we still refer to as ``stable Grothendieck polynomials'' for simplicity---as seen in \cite{lam2007combinatorial}. The standard signed version is due to Buch in \cite{buch2002littlewood}. 
The \emph{stable Grothendieck polynomial} of shape $\mu$ is
\[
G_\mu = \sum_{T} x^T,
\]
where the sum is over set-valued tableaux of shape $\mu$.
Stable Grothendieck polynomials are symmetric functions.

The original definition of stable Grothendieck polynomials, due to Fomin and Kirillov~\cite{fomin1994grothendieck}, can be recovered from Buch's defnition using Hecke insertion, which maps a word $w$ to a pair of tableaux $(P_K(w),Q_K(w))$ of the same ordinary shape where $P_K$ is increasing and $Q_K$ is set-valued.
For a description of Hecke insertion, we refer the reader to~\cite{HeckeInsertion}.

\begin{thm}
\label{thm:groth}
For any increasing tableau $T$ of shape $\mu$,
\[
G_\mu = \sum_{P_K(w) = T} f_{\D(w)}.
\]

\end{thm}

Hecke insertion respects the \emph{$K$-Knuth relations}, which are relations (1-4) of Definition~\ref{def:wkknuth}.
We say two words $v$ and $w$ are \emph{$K$-Knuth equivalent}, denoted $v \equiv w$, if $v$ can be obtained from $w$ by a finite sequence of $K$-Knuth equivalence relations.
\begin{thm}[\cite{BuchandSamuel2013}]
\label{thm:kknuth}
For words $v$ and $w$, if $P_K(v) = P_K(w)$, then $v \equiv w$.

\end{thm}

We are now prepared to show the symmetry of weak shifted stable Grothendieck polynomials.
\begin{thm}
\label{thm:symmetry}
For any shifted shape $\lambda$, $K_\lambda$ is symmetric.

\end{thm}

\begin{proof}
Let $T$ be a URT of shifted shape $\lambda$ and $\mw_T = \{w : \Tab(w) = T\}$.
Since $T$ is a URT, $\mw_T$ is a single weak $K$-Knuth class.
Let $\mathcal{S} = \{P_K(w): w\in \mw_T\}.$
Note $\mathcal{S}$ is finite since it is comprised of increasing tableaux whose largest entry is bounded.
Since the weak $K$-Knuth relations are a refinement of the $K$-Knuth relations, we see by Theorem~\ref{thm:kknuth} that 
\[
\mw_T = \bigcup_{S \in \mathcal{S}} \{P_K(w) = S\}.
\]
Therefore, by Theorems~\ref{kisgood}~and~\ref{thm:groth} we see
\[
K_\lambda = \sum_{w \in \mw_T} f_{\D(w)} = \sum_{P_K(w) \in \mathcal{S}} f_{\D(w)} = \sum_{S \in \mathcal{S}} G_{\mu_S},
\]
where $\mu_S$ is the shape of $S$.
Since stable Grothendick polynomials are symmetric, we see $K_\lambda$ is symmetric.
\end{proof}

\begin{cor}
For any shifted shape $\lambda$, $GP_\lambda$ is symmetric.

\end{cor}

\begin{proof}
The result follows from Proposition~\ref{prop:Ikeda} and Theorem~\ref{thm:symmetry}.
\end{proof}

Explicit decompositions of $K_\lambda$ in terms of $G_\mu$ can be derived as a consequence of forthcoming work by Hamaker, Marberg, Pawlowski, and Young~\cite{hamakerinv2,hamakerinv4}.

\section{Shifted $K$-Poirier-Reutenauer Algebra and a Littlewood-Richardson rule}
\subsection{Poirier-Reutenauer type algebras} In \cite{poirier1995algebres}, Poirier and Reutenauer define a Hopf algebra spanned by the set of standard Young tableaux.
Jing and Li developed a shifted version~\cite{JLshifted}, and Patrias and Pylyavskyy developed a $K$-theoretic analogue~\cite{PPK}. 
In this section, we combine these approaches to introduce a shifted $K$-theoretic analogue.

We say a word $h$ is \textit{initial} if the letters appearing in it are exactly the numbers in $[k]$ for some positive integer $k$. 
For example, the words 54321 and 211345 are initial, but 2344 is not.

Let $[[h]]$ denote the formal sum of all words in the weak $K$-Knuth equivalence class of an initial word $h$:
\[
[[h]] = \sum_{h\ \ksw \ w} w.
\]
This is an infinite sum; however, the number of terms in $[[h]]$ of each length is finite.
For example,
\[
[[213]] = 213 + 231 + 123 + 321 + 3221 + 3321 + 3211+ 32111 + \cdots .
\]
Let $SKPR$ denote the vector space over $\mathbb{R}$ spanned by all sums of the form $[[h]]$ for some initial word $h$.  
We will define a product on $SKPR$ to make it an algebra. 
To do so, we need the following lemma. 
\begin{lem}\label{finitesumoftableau}
We have 
\[[[h]] = \sum_T\left( \sum_{\Tab(w) = T} w \right),
\]
where the sum is over all increasing shifted tableaux $T$ whose reading word is in the weak $K$-Knuth equivalence class of $h$.  
\end{lem}
\begin{proof}
The result follows from Corollary~\ref{insertionimplieskknuth}.
\end{proof}
Note that the set of tableaux we sum over is finite by Lemma~\ref{finitetableau}. 

\subsection{Product structure} 
Let $\shuffle$ denote the usual shuffle product of words and  $w[n]$ be the word obtained from $w$ by increasing each letter by $[n]$. 
For example, if $w=312$, $w[4]=756$. 
For $h$ a word in the alphabet $[n]$ and $h\pr$ a word in the alphabet $[m]$, we define the product in $SKPR$ by
\[
[[h]] \cdot [[h\pr]] = \sum_{w \ksw h, w\pr \ksw h\pr} w \shuffle w\pr[n].
\]
For example
\[[[12]]\cdot[[1]] = [[123]]+[[312]]+[[3123]].
\]
In general, $[[h]] \cdot [[h\pr]]$ can be expressed as a sum of classes.
\begin{thm}\label{productissum}
The product of any two initial words $h$ and $h\pr$ can be written as
\[
[[h]] \cdot [[h\pr]] = \sum_{h\pr\pr}\  [[h\pr\pr]],
\]
where the sum is over some set of initial words $h\pr\pr$.  
\end{thm}
\begin{proof}
By Lemma \ref{interval}, we know that once a word appears in the right-hand sum, its entire equivalence class must also appear.  
The result then follows.  
\end{proof}

By Lemma~\ref{finitesumoftableau}, we can write $[[h]]$ as a sum over tableaux, which we know can be sorted into finite equivalence classes.
This implies that we can write the product as an explicit sum over sets of tableaux.

\begin{thm}
\label{productistableausum}
Let $h$ be a word in alphabet $[n]$, and let $h\pr$ be a word in alphabet $[m]$.  
For $\mathcal{T}= \{\Tab(w): w \  \hat{\equiv} \ h\}$, we have
\[
[[h]] \cdot [[h\pr]] = \sum_{T \in \mathcal{T}(h\shuffle h\pr)} \sum_{\Tab(w) = T} w,
\]
where $\mathcal{T}(h \shuffle h\pr)$ is the set of shifted tableaux $T$ such that $T|_{[n]} \in \mathcal{T}$ and $\row(T)|_{[n+1,n+m]}\ \ksw\ h\pr[n].$
\end{thm}

\begin{proof}
If $w$ is a shuffle of some $w_1\ \ksw\ h$ and $w_2\ \ksw \ h\pr[n]$, then by Lemma \ref{intervaltableau} $\Tab(w)|_{[n]} = \Tab(w_1) \in \mathcal{T}$. 

Since $\Tab(w)\ \ksw \ T_w$ by Theorem~\ref{jdtgivesinsertion}, we see that $\Tab(w)|_{[n+1,n+m]}\ \ksw \ T_w|_{[n+1,n+m]}$. This implies $\row(\Tab(w)|_{[n+1,n+m]})\ \ksw \ h'[n]$.
Now using Lemma \ref{finitesumoftableau} and Theorem \ref{productissum}, we see that the product can be expanded in this way.  
\end{proof}

By Lemma~\ref{finitetableau}, the set $T(h \shuffle h\pr)$ is finite since all of its tableaux are on the finite alphabet $[n+m]$.  
\begin{ex}
Consider $h=12$ and $h\pr=12$.  The set $T(12 \shuffle 12)$ consists of the seven tableaux below. 
\begin{center}
\ytableausetup{boxsize=.5cm}
\begin{ytableau}
1 & 2 & 3 & 4
\end{ytableau} \hspace{.2in}
\begin{ytableau}
1 & 2 & 3\\
\none & 4
\end{ytableau} \hspace{.2in}
\begin{ytableau}
1 & 2 & 4 \\
\none & 3
\end{ytableau}\hspace{.2in}
\begin{ytableau}
1 & 2 & 3 & 4 \\
\none & 3
\end{ytableau}\hspace{.2in}
\begin{ytableau}
1 & 2 & 3 & 4 \\
\none & 4
\end{ytableau}\\
\vspace{.1in}
\begin{ytableau}
1 & 2 & 3  \\
\none & 3 & 4
\end{ytableau}\hspace{.2in}
\begin{ytableau}
1 & 2 & 3 & 4 \\
\none & 3 & 4
\end{ytableau}
\end{center}
For each of these tableaux, restricting to the alphabet $\{1,2\}$ gives the tableau $\Tab(12)$.  Also, the reading word of each restricted to the alphabet $\{3,4\}$ is weak $K$-Knuth equivalent to $h\pr[2]=34$.  
\end{ex}

\begin{cor}\label{productclosed}
The vector space $SKPR$ is closed under the product operation.  
That is, the product of two classes can be expressed as a finite sum of classes.  
\end{cor}
\begin{proof}
By Theorem~\ref{productissum}, we know that the product of two words can be expressed as a sum of weak $K$-Knuth equivalence classes.  
From Lemma~\ref{finitesumoftableau}, each of these classes can be written as a finite sum of tableau equivalence classes.  
That is to say, while weak $K$-Knuth classes are coarser than insertion tableau classes, there are still only finitely many.  
Theorem~\ref{productistableausum} tells us that the product of two words can be written as a finite sum of tableau equivalence classes, so the expansion in Theorem~\ref{productissum} must be finite.  
\end{proof}

Since the product in $SKPR$ is defined in terms of equivalence classes of tableaux, it follows that classes corresponding to URTs should have particularly simple products.

\begin{thm}\label{URTproduct}
Let $T_1$ and $T_2$ be two URTs. Then
\[
\left(\sum_{\Tab(u)= T_1} u\right)\cdot \left(\sum_{\Tab(v)= T_2}v\right)= \sum_{T\in T(T_1\shuffle T_2)}\sum_{\Tab(w)=T}w,
\]
where $T(T_1\shuffle T_2)$ is the finite set of shifted tableaux $T$ such that $T|_{[n]}=T_1$ and \\ $\Tab(\row(T)|_{[n+1,n+m]})=T_2$.
\end{thm}

\begin{proof}
Since $T_1$ and $T_2$ are URTs, the left-hand side is $[[\row(T_1)]]\cdot[[\row(T_2)]]$, which equals the right-hand side by Theorem~\ref{productistableausum}. 
\end{proof}

\begin{remark}\label{rem:nocoproduct}
In the Poirier-Reutenauer Hopf algebra, its $K$-theoretic analogue, and its shifted analogue, one defines a coproduct by first defining a coproduct on words: 
\[
\blacktriangle(w_1\ldots w_n)=\displaystyle \sum_{i=0}^n  \std(w_1\ldots w_i) \otimes \std(w_{i+1}\ldots w_n).
\]
Here, $\std(w)$ is the unique initial word whose letters are in the same relative order as the letters in $w$. 
One then extends this coproduct to equivalence classes by defining $\Delta([[w]])=\sum \blacktriangle(h)$, where we sum over all $h$ equivalent to $w$. 
The right-hand side can be expressed as a sum of tensors of equivalence classes: $\Delta([[w]])=\sum [[u]]\otimes [[v]]$. 
This is necessary for the coproduct to be well-defined in these vector spaces, and is not true in the case of $SKPR$. 

Indeed, consider the the word 1232. The term $\std(1)\otimes \std(232)=1\otimes 121$ appears in $\blacktriangle(1232)$, so if $\Delta([[1232]])$ could be expressed as a sum of tensors of weak $K$-Knuth equivalence classes, it should contain the term $1 \otimes 211$. 
The words with letters exactly $[3]$ that have $1 \otimes 211$ in their coproduct are 1322, 2311, and 3211.
Since $\Tab(1232)$ is a URT, it is easy to see that 1322, 2311, and 3211 are not weak $K$-Knuth equivalent to 1232 by comparing insertion tableaux. 
We conclude that $1\otimes 121$ appears in $\Delta([[1232]])$ but $ 1 \otimes 211$ does not, which implies that the standard coproduct does not work in $SKPR$.
\end{remark}

We now construct an algebra homomorphism that takes a weak $K$-Knuth equivalence class $[[h]]$ to a sum of fundamental quasisymmetric functions. We will use this map to prove a Littlewood-Richardson rule for the $K_\lambda$.

\begin{thm}
The linear map $\phi:$SKPR$\to$ QSym defined by 
\[
[[h]] = \sum_{w\ksw h} f_{\D(w)}
\]
is an algebra homomorphism.
\end{thm}

\begin{proof}
We show that the map preserves the product. 
First, \cite[Proposition~5.9]{lam2007combinatorial} implies
\[
f_{\D(w\pr)}\cdot f_{\D(w\pr\pr)} = \sum_{w\in Sh(w\pr,w\pr\pr[n])} f_{\D(w)},
\]
where the sum is over all shuffles of $w\pr,w\pr\pr[n]$. 
It follows that
\[\begin{split}
\phi([[h_1]]\cdot[[h_2]]) &= \phi\left(\sum_{w\pr\ksw h_1,w\pr\pr\ksw h_2}w\pr\shuffle w\pr\pr\right)\\
&=\sum_{w\pr\ksw h_1,w\pr\pr\ksw h_2} \quad \sum_{w\in Sh(w\pr,w\pr\pr[n])}f_{\D(w)}\\
&=\left(\sum_{w\pr\ksw h_1}f_{\D(w\pr)}\right)\left(\sum_{w\pr\pr\ksw h_2}f_{\D(w\pr\pr)}\right)=\phi([[h_1]])\phi([[h_2]]).
\end{split}
\]
\end{proof}

\begin{thm}\label{mapisksum}
Letting $\lambda(T)$ denote the shape of $T$, we have
\[
\phi([[h]]) = \sum_{\row(T)\ksw h} K_{\lambda(T)}.
\]
\end{thm} 

\begin{proof}
By Theorem \ref{kisgood}, we see
\[
\phi([[h]]) = \sum_{w \ksw  h}f_{\D(w)} = \sum_{T\ksw \Tab(h)}\sum_{\Tab(w)=T}f_{\D(w)} = \sum_{T \ksw \Tab(h)}K_{\lambda(T)} = \sum_{\row(T) \ksw  h} K_{\lambda(T)}.
\]
\end{proof}

With the algebra homomorphism $\phi$ defined above, we can show the Littlewood-Richardson rule for $\symfcn$ by using Theorem \ref{mapisksum}.
\begin{thm}\label{LR product}
Let $T$ be a URT of shape $\mu$. Then we have
\[
K_\lambda K_\mu = \sum_\nu c^\nu_{\lambda,\mu}K_\nu,
\]
where $c^\nu_{\lambda,\mu}$ is given by the number of increasing shifted skew tableaux $R$ of shape $\nu/\lambda$ such that $\Tab(\row(R))=T$.
\end{thm}

\begin{proof}
In addition to $T$, fix a URT $T\pr$ of shape $\lambda$. Then by Theorems \ref{URTproduct} and \ref{mapisksum}, we have
\[
\begin{split}
K_\lambda K_\mu &= \phi([[\row(T\pr)]])\phi([[\row(T)]])\\
&= \phi([[\row(T\pr)]]\cdot[[\row(T)]])\\
&= \sum_{R\in T(T\pr\shuffle T)}\sum_{\Tab(w)=R}w\\
&= \sum_{R\in T(T\pr\shuffle T)} K_{\lambda(R)},
\end{split}
\]
where $T(T\pr\shuffle T)$ is the set of shifted tableaux $R$ such that $R|_{[|\lambda |]}=T\pr$ and \\ $\Tab(\row(R|_{[|\lambda |+1,|\lambda |+|\mu |]}))=T$, giving our result.
\end{proof}

This rule, up to sign, coincides with the rules of Clifford, Thomas and Yong~\cite[Theorem 1.2]{clifford2014k} and Buch and Samuel~\cite[Corollary 4.8]{BuchandSamuel2013}.
Theorem~\ref{t:ktheory} follows immediately.

\section*{Acknowledgments}
This research was part of the 2015 REU program at the University of
Minnesota, Twin Cities, and was supported by RTG grant NSF/DMS-1148634. We would like to thank Pasha
Pylyavskyy and Vic Reiner for their mentorship and support. 

\bibliographystyle{plain}
\bibliography{main}

\end{document}